\documentclass[12pt,a4paper,reqno]{amsart}
\pdfoutput=1
\usepackage{amsfonts}
\usepackage{amsthm}
\usepackage{amsmath}
\usepackage{amssymb}
\usepackage{mathabx}
\usepackage{bm}
\usepackage{amscd}
\usepackage[latin2]{inputenc}
\usepackage{t1enc}
\usepackage[mathscr]{eucal}
\usepackage{indentfirst}
\usepackage{graphicx}
\usepackage{graphics}
\usepackage{pict2e}
\usepackage{epic}
\usepackage{cases}
\numberwithin{equation}{section}
\usepackage[margin=2.9cm]{geometry}
\usepackage{epstopdf} 
\usepackage[colorlinks,linkcolor=blue]{hyperref}
\usepackage{color}
\usepackage{todonotes}
\usepackage[capitalise,noabbrev]{cleveref}
\usepackage{cancel}

\crefformat{equation}{(#2#1#3)}
\crefrangeformat{equation}{(#3#1#4) to~(#5#2#6)}

\usepackage[normalem]{ulem}

\allowdisplaybreaks

\theoremstyle{plain}
\newtheorem{Thm}{Theorem}[section]
\newtheorem*{Thm*}{Thm}
\newtheorem{Lem}[Thm]{Lemma}
\newtheorem{Cor}[Thm]{Corollary}

\theoremstyle{definition}

\newtheorem{Rem}[Thm]{Remark}
\newtheorem{?}[Thm]{Problem}

\newcommand{\R}{{\mathbb R}}

\newcommand{\A}{{\mathcal A}}

\newcommand{\Torus}{\mathbb{T}}

\newcommand{\p}{\partial}

\newcommand{\lap}{\triangle}
\newcommand{\ut}{\tilde{u}}
\newcommand{\ub}{\overline{u}}

\newcommand{\norm}[1]{\left\lVert#1\right\rVert}
\newcommand{\abs}[1]{\left\lvert#1\right\rvert}

\author[F. Huang and Q. Yuan]{Feimin Huang$^{1,2}$ ~~~~Qian Yuan$^{1,*}$\\[2cm]
{\tiny	Dedicated to the 80th Birthday of Ling Hsiao}}

\address{$^1$Academy of Mathematics and Systems Science,  CAS, Beijing 100190, China}
\address{$^2$ School of Mathematical Sciences, University of Chinese Academy of Sciences, Beijing 100049, China}

\thanks{*Corresponding Author}

\thanks{Feimin Huang is partially supported by NSFC Grant No. 11371349 and 11688101.}
\email{fhuang@amt.ac.cn}

\thanks{Qian Yuan is partially supported by the China Postdoctoral Science Foundation funded project	(2019M660831).}
\email{qyuan103@link.cuhk.edu.hk}

\title[Periodic perturbations]{Stability of planar rarefaction waves for scalar viscous conservation law under periodic perturbations}

\begin{document}

\begin{abstract}
The large time behavior of the solutions to a multi-dimensional viscous conservation law is considered in this paper. It is shown that the solution time-asymptotically tends to the planar rarefaction wave if the initial perturbations are multi-dimensional periodic. The time-decay rate is also obtained. Moreover, a Gagliardo-Nirenberg type inequality is established in the domain $ \R\times \mathbb T^{n-1} (n\geq2) $, where $\mathbb T^{n-1}$ is the $ n-1 $-dimensional torus.
\end{abstract}

\maketitle

\textbf{\small AMS Subject Classification}: {\small 35B65, 35L65, 35Q35}

{\small Keywords: Planar rarefaction wave, periodic perturbation, viscous conservation law}
\section{introduction}

We are concerned with  a scalar viscous conservation law, which reads in $\R^n$ as 
\begin{equation}{}\label{eq}
	 \p_t u(x,t) + \sum\limits_{i=1}^{n} \p_i \left(f_i(u(x,t))\right) = \lap u(x,t), \quad t>0, x \in \mathbb{R}^n, 
	\end{equation}
where  $ u(x,t) \in \R, $ $ \p_i := \p_{x_i} ~(i=1,2,..,n) $ and $ \lap = \sum\limits_{i=1}^n \p_i^2; $ the fluxes $f_i(u) ~ (i=1,2,..,n) $ are smooth and $ f_1''(u)\geq a_0 $ for some constant $ a_0>0$.

It is well known \cite{Lax1957,Smoller1994} that the one-dimensional (1-d) conservation law, i.e. the right hand side of \eqref{eq} is zero and $ n=1, $ has rich wave phenomena including shocks and rarefaction waves. 
A centered rarefaction wave $ u^R(x_1,t) $ is an entropy solution of the following Riemann problem 
\begin{equation}\label{eq-1}
\begin{cases}
\p_t u^R(x_1,t) + \p_1 f_1(u^R(x_1,t)) = 0,   \\
u^R(x_1,0) = \left\{\begin{array}{ll} \bar{u}_l, ~ x_1<0,\\
\bar{u}_r, ~x_1>0,
\end{array}
\right.
\end{cases}
\end{equation}
with $\bar{u}_l<\bar{u}_r$, and has an explicit formula as
\begin{equation}\label{eq-2}
u^R(x_1,t) = \begin{cases}
\bar{u}_l, & x_1 \leq f'_1(\bar{u}_l)t, \\
(f_1')^{-1}(\frac{x_1}{t}), \quad &  f'_1(\bar{u}_l)t<x_1 \leq f'_1(\bar{u}_r)t,\\
\bar{u}_r, & x_1>f'_1(\bar{u}_r)t.
\end{cases} 
\end{equation}
It is shown in \cite{IO1960,Liu1978,Xin1990,KNM2004} that the rarefaction wave given in \eqref{eq-2} is asymptotically stable for both the inviscid and viscous conservation laws in $L^2$ framework provided that the initial perturbations are $ L^2 $ integrable. 
The analysis for the solutions of conservation laws with periodic initial data is important,  cf. \cite{Lax1957,GL1970,Dafermos1995,Dafermos2003,XYY2019-1,XYY2019-2,YY2019}. 
In particular, the large time behavior of the nonlinear waves under periodic perturbations for viscous conservation laws is firstly 
investigated in \cite{XYY2019-1}  with the aid of maximum principle.  
The 1-d rarefaction wave $u^R(x_1,t)$ becomes the planar rarefaction wave in the multi-dimensional (m-d) case. The main purpose of this paper is to use the energy method to extend 
the work of \cite{XYY2019-1} to the m-d case, i.e., 
to show whether the planar rarefaction wave  
is stable or not under m-d periodic perturbations. 
Besides, we refer to \cite{HH2020, IO1960, KM1985, KMN1985, MN1985, MN1986, Goodman1986, LX1988, LX1997,HXY2008, LZ2009} and the references therein for the other interesting works about the stability of shocks, rarefaction waves and contact discontinuities.

\vspace{0.2cm}

Now we formulate the main result. Since the centered rarefaction wave given in \eqref{eq-2} is a weak solution and only Lipschitz continuous with respect to $x_1,$ we need to construct a viscous version of rarefaction wave $ \ut^R(x_1,t) $ to replace the original one. Following \cite{MN1986,Xin1990}, $ \ut^R(x_1,t) $ can be constructed as a smooth solution of the viscous conservation law \eqref{eq}, i.e.,
\begin{equation}\label{approx}
\begin{cases}
\partial_t \ut^R +\partial_1 (f_1(\ut^R))= \p_1^2 \ut^R, & \\
\ut^R(x_1, 0) = \frac{\ub_l+\ub_r}{2}+\frac{\ub_r-\ub_l}{2}\frac{e^{x_1}-e^{-x_1}}{e^{x_1}+e^{-x_1}}. &
\end{cases}
\end{equation}
It is straightforward to check that 
\begin{equation*}
\lim_{x_1\to -\infty} \ut^R(x_1,0)= \ub_l < \ub_r = \lim_{x_1\to +\infty} \ut^R(x_1,0).
\end{equation*}
Consider the scalar equation \eqref{eq} with the following initial data
\begin{equation}
 u(x, 0) = u_{0}(x) = \ut^R(x_1,0) + w_0(x), \qquad x \in \mathbb{R}^n,  \label{ic}	
\end{equation}
where  $ w_0(x)$ is a m-d periodic function defined on the n-d torus $ \Torus^n = \prod\limits_{i=1}^{n} [0, 1]. $ Without loss of generality (by adding the average constant onto $ \ub_l $ and $ \ub_r, $ respectively), one can assume that 
\begin{equation}\label{ave-0}
\int_{\Torus^n} w_0(x)dx = 0.
\end{equation}

We are ready to state the main theorem.
\begin{Thm}\label{Thm}
	Assume that the periodic perturbation $w_0(x) \in H^{\left[\frac{n}{2}\right] + 2}(\Torus^n) $ and satisfies \cref{ave-0}.
	
	Then there exists a unique global smooth solution $ u $ of \cref{eq}, \cref{ic} satisfying
	\begin{equation}\label{rate}
	\norm{u(x,t) - \ut^R(x_1,t)}_{L^\infty(\R^n)}\leq C (1+t)^{-\frac{1}{2}}, \quad t>0,
	\end{equation}
	where the constant $ C>0 $ is independent of $ t. $
\end{Thm}	

\begin{Rem}\label{Rem-ic}
It is shown in 
\cite{KNM2004} that any viscous rarefaction waves connecting same end states are time-asymptotically equivalent in the $ L^\infty(\R) $ space with the rate $ t^{-\frac{1}{2}}. $ We remark that the stability result \cref{rate} also holds true for more general initial values $ \ut^R(x_1,0) + v_0(x_1) + w_0(x) $ instead of \cref{ic}, where $ v_0(x_1) $ is any 1-d function in $ L^1(\R) \cap H^1(\R). $
\end{Rem}

\begin{Rem}
	Theorem \ref{Thm} is the first work concerning the m-d periodic perturbations around the nonlinear waves for conservation laws, which shows that the oscillations in all directions around the planar rarefaction wave decay to zero with the rate $t^{-\frac12},$ even though the initial perturbations keep oscillating at infinity $ \abs{x} \rightarrow +\infty$. In other words, the oscillations are eliminated due to the genuine nonlinearity of the equation.
\end{Rem}

	

Let us  outline the proof of \cref{Thm}. Motivated by  \cite{Xin1990} and \cite{KNM2004} in which the planar rarefaction waves are shown to be stable under m-d perturbations for the scalar equation \cref{eq}, we want to use the energy method to prove \cref{Thm}. 
However, the initial periodic perturbation $w_0(x)$ is not integrable on $\R^n,$ and has no any limit at far fields.  Thus, the effective  energy method developed in the previous articles can not be applied here directly.  
To overcome this difficulty, we construct a suitable ansatz $\ut(x,t)$ which contains the oscillations in the $x_1$ direction so that the oscillations in the difference between the solution $u$ and the ansatz $\ut$ is eliminated in the $x_1$ direction, i.e., $(u-\ut)(x,t) \in L^1_{x_1}(\R), $  This is the key point in our proof. Then  the stability with time-decay rates \cref{rate} can be obtained by the $L^p$ energy method developed in \cite{KNM2004}. 

To prove \cref{Thm}, we establish a Gagliardo-Nirenberg type inequality in the domain $\Omega:= \R\times \mathbb T^{n-1}$, which is the second novelty in this paper. The Gagliardo-Nirenberg (GN) inequality 
is very useful in the field of PDEs and usually holds in the whole space $\mathbb R^n$ or in the domain with zero Dirichlet boundary condition. Let us recall the GN inequality which reads as, for any $ 1\leq p\leq +\infty $ and integer $ 0\leq j < m $, 
\begin{equation}\label{gn}
\norm{\nabla_x^j u}_{L^p(\mathbb R^n)} \leq C \norm{\nabla_x^m u}_{L^r(\mathbb R^n)}^{\theta} \norm{u}_{L^q(\mathbb R^n)}^{1-\theta},
\end{equation}
where $ \frac{1}{p} = \frac{j}{n} + \left(\frac{1}{r} -\frac{m}{n} \right) \theta + \frac{1}{q} \left(1-\theta \right) $, $ \frac{j}{m} \leq \theta \leq 1$ and $C$ is a constant independent of $u$. However, the $ n $-d GN inequality \cref{gn} does not hold in the domain $\Omega$ generally owing to the following counterexample.

\vspace{0.1cm}

\textbf{Counterexample: }  
It is noted that any 1-d function $ 0\neq f(x_1) \in C_c^\infty(\R) $ is periodic in the $ x_i $ direction for $ i=2,\cdots,n $ (one cannot exclude the 1-d case since the initial perturbations satisfying \cref{ave-0} include the 1-d periodic functions). Then corresponding to the case $j=0, p=\frac{n}{n-1}, \theta=m=r=1$ in \cref{gn}, which is exactly the $ n $-d Sobolev inequality, one can let $ \zeta_d(x_1) = f(d^{-1} x_1), $ $ d>0, $ and a direct computation implies that
\begin{equation}\label{exam}
\norm{\zeta_d}_{L^{\frac{n}{n-1}}(\Omega)}=C_{n,d}\norm{\nabla_x \zeta_d}_{L^1(\Omega)},
\end{equation}
where $ C_{n,d}:= d^{\frac{n-1}{n}} \norm{f}_{L^{\frac{n}{n-1}}(\R)}/\norm{f'}_{L^1(\R)} $. Since $C_{n,d}\to +\infty$ as $ d \rightarrow +\infty$, we conclude that in general the $ n $-d GN inequality \eqref{gn} is not true in $\Omega$ without additional  conditions.

\vspace{0.2cm}

Instead, based on a function-decomposition in \cref{Lem-decompose} below, we establish a Gagliardo-Nirenberg type inequality on the domain $ \Omega$.

\begin{Thm}[GN type inequality on $ \Omega=\R\times\Torus^{n-1} $]\label{Thm-G-N}
	Let $ u \in L^q(\Omega) $ and $ \nabla^m u \in L^r(\Omega) $ where $ 1\leq q,r\leq +\infty $ and $ m\geq 1, $ and $ u $ is periodic in the $ x_i $ direction for $ i=2,\cdots,n. $ Then there exists a decomposition $ u(x) = \sum\limits_{k=0}^{n-1} u^{(k)}(x) $ such that each $ u^{(k)} $ satisfies the $ k+1 $-dimensional GN inequality \cref{gn}, i.e., 
	\begin{equation}\label{G-N-type-1}
	\norm{\nabla^j u^{(k)}}_{L^p(\Omega)} \leq C \norm{\nabla^{m} u}_{L^{r}(\Omega)}^{\theta_k} \norm{u}_{L^{q}(\Omega)}^{1-\theta_k},
	\end{equation}
	for any $ 0\leq j< m $ and $ 1\leq p \leq +\infty $ satisfying $ \frac{1}{p} = \frac{j}{k+1} + \left(\frac{1}{r}-\frac{m}{k+1}\right) \theta_k + \frac{1}{q}\left(1-\theta_k\right) $ and $ \frac{j}{m} \leq \theta_k \leq 1, $ 
	with the following exceptional cases,
	\begin{itemize}
		\item[1) ] if $ j=0, r m <k+1 $ and $ q = +\infty, $ additional assumption that $ u \to 0 $ as $ \abs{x_1} \to +\infty$ is required;
		
		\item[2) ] if $ 1<r<+\infty $ and $ m-j-\frac{k+1}{r} $ is a non-negative integer, additional assumption that $ \theta_k<1$ is required.
	\end{itemize}
	Hence, it holds that 
	\begin{equation}\label{G-N-type-2}
	\norm{\nabla^j u}_{L^p(\Omega)} \leq C \sum_{k=0}^{n-1}\norm{\nabla^{m} u}_{L^{r}(\Omega)}^{\theta_k} \norm{u}_{L^{q}(\Omega)}^{1-\theta_k},
	\end{equation}
	where the constant $ C>0 $ is independent of $ u. $
\end{Thm}

\begin{Rem}
	It is noted that the first term for $ k=0 $ on the right-hand side (RHS) of \cref{G-N-type-2} is necessary since $ u $ can be a 1-d function defined on $ \R. $ 
	For example, considering the same function $\zeta_d(x_1)$ as in the counterexample, when corresponding to the case $j=0, p = r=2 $ and $ m=q =1$ in \eqref{G-N-type-2}, a direct computation gives that
	\begin{equation}\label{couter}
	\norm{\zeta_d}_{L^2(\Omega)} = C d^{ \frac{1}{2} \left(3\theta-1\right) }
	\norm{\nabla\zeta_d}^{\theta}_{L^2(\Omega)}\norm{\zeta_d}^{1-\theta}_{L^1(\Omega)} \quad \text{for any } 0\leq \theta \leq 1,
	\end{equation}
	where $ C>0 $ is independent of $ d. $ The estimate \cref{G-N-type-2} can be true only for $\theta=\theta_0=\frac13 $. Otherwise, one can let $ d\to 0+ $ if $ \theta<\frac13 $ and $ d\to +\infty $ if $ \theta>\frac13, $ respectively to get the contradiction. 
\end{Rem}

\begin{Rem}
	An interesting interpolation inequality involving $\nabla^2 u$  in the domain $\mathbb R\times \mathbb T^2$ was established  in \cite{WW2019}.
\end{Rem}

The rest of the paper is organized as follows. Some preliminaries 
on the viscous rarefaction waves and the construction of the ansatz are given in \cref{pre}. 
\cref{Sec-inter} is devoted to the proof of \cref{Thm-G-N} and an interpolation inequality, \cref{Cor-inter}. 
In \cref{Sec-Apriori}, we show the desired a priori estimates, \cref{Thm-apriori}; thereafter, the proof of \cref{Thm} is completed.

\section{Preliminaries and Ansatz}\label{pre}

Some properties of the viscous rarefaction wave are listed as follows. 

\begin{Lem}\label{Lem-ut-R}
	The viscous rarefaction wave $ \ut^R(x_1,t) $ solving \cref{approx} satisfies
	\begin{equation}\label{lim-equiv}
	\lim\limits_{t\rightarrow +\infty} \norm{\left(\ut^R - u^R\right)(\cdot,t)}_{L^\infty(\R)} = 0.
	\end{equation}
	Moreover, for any $ t>0 $ it holds that
	\begin{align}
	& 0 < \p_1 \ut^R(x_1,t) \leq \min \left\{ \frac{C}{t},~ \max_{x_1\in\R} \left(u_0^R\right)' \right\} \qquad \forall x_1 \in \R, \label{entropy} \\
	& \norm{ \left(\ut^R(\cdot,t) - \ub_l \right) \left(\ut^R(\cdot,t) - \ub_r \right) }_{L^1(\R)} \leq C(1+t), \label{ut-1} \\
	& \norm{\p_1 \ut^R(\cdot,t)}_{L^p(\R)} \leq Ct^{-1+\frac{1}{p}}, \qquad p \in [1,+\infty), \label{ut-2}
	\end{align}
	where the constant $ C>0 $ is independent of $ t. $
\end{Lem}
\begin{proof}
	The $ L^\infty $ time-asymptotically equivalence \cref{lim-equiv} has been proved in \cite{IO1960,KNM2004}.
	By the maximum principles for \cref{approx} and the equation of the derivative $ v = \p_1 \ut^R, $
	\begin{equation}
	\p_t v -\p_1^2 v + f_1''(\ut^R) v^2 + f_1'(\ut^R) \p_1 v =0,
	\end{equation}
	one can get that $ \sup\limits_{x_1\in\R,t\geq 0}\abs{\ut^R(x_1,t)} \leq \norm{u_0^R}_{L^\infty(\R)} $ and $ 0< \p_1 \ut^R(x_1,t) \leq \max\limits_{x_1\in\R} \left(u_0^R\right)' (x_1). $
	Besides, it was shown in \cite{Oleinik1957} that the solution $ \ut^R $ of the 1-d convex conservation law \cref{approx} satisfies the well-known Ole\v{\i}nik entropy condition, i.e., $ \p_1 \ut^R(x_1,t) \leq \frac{C}{t} $ for all $ x_1\in\R, $ where $ C>0 $ is independent of $ t. $ 
	
	Since $ \ub_l < \ut^R < \ub_r $ and both $ \ut^R $ and $ \p_1 \ut^R $ are uniformly bounded, it follows from \cref{approx} that
	\begin{equation*}
	\frac{d}{dt} \left[\int_{-\infty}^0 \left(\ut^R(x_1,t)-\ub_l\right) dx_1 + \int_0^{+\infty} \left(\ub_r-\ut^R(x_1,t)\right) dx_1 \right] \leq C, \quad t>0,
	\end{equation*}
	which implies that
	\begin{equation*}
	\int_{-\infty}^0 \left(\ut^R(x_1,t)-\ub_l\right) dx_1 + \int_0^{+\infty} \left(\ub_r-\ut^R(x_1,t)\right) dx_1 \leq C(1+t), \quad t>0,
	\end{equation*}
	and thus \cref{ut-1} holds true. And \cref{ut-2} can follow from
	\begin{equation*}
	\norm{\p_1 \ut^R}_{L^p(\R)}^p \leq \norm{\p_1 \ut^R}_{L^\infty(\R)}^{p-1} \int_\R \p_1 \ut^R(x_1,t) dx_1 \leq Ct^{-p+1}.
	\end{equation*}
\end{proof}

Now we construct the ansatz. Set 
\begin{equation}\label{g}
g(x_1,t) := \frac{\ut^R(x_1,t)-\ub_l}{\ub_r-\ub_l}, \quad x_1\in\R, ~t\geq 0,
\end{equation}
which is smooth, belongs to the interval $(0,1) $ and satisfies $ \p_1 g(x_1,t) >0 $ for any $ x_1\in\R, t\geq 0. $

Let $ u_l(x,t) $ and $ u_r(x,t) $ denote the two periodic solutions of \eqref{eq} (see \cite{Serre2002} for the global existence) with the respective periodic data
\begin{equation}\label{ic-per}
u_l(x,0) = \ub_l + w_0(x) \quad \text{ and} \quad u_r(x,0) = \ub_r + w_0(x), \quad x\in\R^n.
\end{equation} 
And define
\begin{equation}\label{phi-l-r}
w_i(x,t) := u_i(x,t) - \ub_i, \quad i = l \text{ or } r,
\end{equation}
which is a periodic function with zero average for any $ t\geq 0. $ Then we have 
\begin{Lem}\label{Lem-per}
	If $ w_0(x) \in H^{\left[\frac{n}{2}\right] + 2}(\Torus^n) $ satisfies \cref{ave-0}, then for $ i = l $ or $ r, $ it holds that
	\begin{equation}
	\norm{w_i(\cdot,t)}_{W^{1,\infty}(\R^n)}
	\leq C \norm{w_0}_{ H^{\left[\frac{n}{2}\right] + 2}(\Torus^n) } e^{- 2 \alpha t}, \quad t\geq 0, 
	\end{equation}
	where the constants $ \alpha >0 $ and $ C>0 $ are independent of $ t. $
\end{Lem}
\begin{proof}
	The proof of \cref{Lem-per} is based on basic $ L^2 $ energy  estimates and the Poinc\'{a}re inequality on $ \Torus^n, $ which is standard and thus omitted.
\end{proof}

We are ready to  construct the ansatz as follows:
\begin{equation}\label{ansatz}
\ut(x,t) := u_l(x,t) \left(1-g(x_1,t)\right) + u_r(x,t) g(x_1,t),
\end{equation}
which is periodic in the $ x_i $ direction for $ i=2,\cdots,n. $
By direct calculations, the source term induced by the ansatz $ \ut(x,t) $ is given as follows.
\begin{align}
h(x,t) :=~ & \p_t \ut + \sum\limits_{i=1}^{n}\p_i \left(f_i(\ut(x,t))\right) - \lap \ut(x,t) \notag \\
=~ & \left(u_r-u_l\right) g(1-g) \sum\limits_{i=1}^n \left[ \sigma_i(u_l,\ut) \p_i w_l - \sigma_i(u_r,\ut) \p_i w_r \right] \label{h} \\
& + \left(u_r-u_l\right) \sigma_1(\ut^R,\ut) \left(\ut-\ut^R\right) \p_1 g - 2\p_1\left(w_r-w_l\right) \p_1 g, \notag
\end{align}
where $ \sigma_i(u,v) := \int_0^1 f_i''(v+\theta(u-v)) d\theta, ~i=1,\cdots,n. $ Thus, the source term $ h(x,t) $ is also periodic in the $ x_i $ direction for $ i=2,\cdots,n$. 

\begin{Lem}\label{Lem-h}
	Under the assumptions of \cref{Lem-per}, it holds that for any $ p\in[1,+\infty], $
	\begin{equation}\label{est-h}
	\norm{h(\cdot,t)}_{L^p(\Omega)} \leq C \norm{w_0}_{ H^{\left[\frac{n}{2}\right] + 2}(\Torus^n) } e^{-\alpha t}, \quad t\geq 0, 
	\end{equation}
	where $ \alpha >0 $ is the constant in \cref{Lem-per}, and $ C>0 $ is independent of $ t. $
\end{Lem}
\begin{proof}
	By \cref{ut-1}, \cref{ut-2} and \cref{Lem-per}, it is straightforward to check that \cref{est-h} holds for $ p=1 $ and $ +\infty $. And for $ p\in (1,+\infty), $  \cref{est-h} follows from the interpolation $ \norm{h}_{L^p} \leq \norm{h}_{L^\infty}^{1-\frac{1}{p}} \norm{h}_{L^1}^{\frac{1}{p}}. $ 
\end{proof}

\section{Interpolation inequality on $\Omega = \mathbb R\times \mathbb T^{n-1}$}\label{Sec-inter}

Although the $ n $-d GN inequality \eqref{gn} does not hold in the domain $\Omega=\R\times\Torus^{n-1}$ generally, it is true with an additional condition, that is, 
\begin{Lem}\label{Lem-inter}
	Let $ u \in L^q(\Omega), $ and its derivatives of order $ m, $ $ \nabla^m u \in L^r(\Omega), $ where $ m \geq 1 $ and $ 1\leq r,q \leq +\infty. $ And assume that $ u $ is periodic in the $ x_i $ direction and satisfies 
	\begin{equation}\label{ave-0-i}
	\int_{\Torus} u(x) dx_i = 0 \quad \text{ for all } i=2,3,\cdots,n.
	\end{equation}
	Then the $ n $-d GN inequality \eqref{gn} holds true for $u$.
\end{Lem}

\begin{proof}
	Following \cite[Lecture II]{Nirenberg1959}, it suffices to prove the following two extreme cases.
	\begin{itemize}
		\item[1) ] For a.e. $ x\in \Omega, $ it holds that
		\begin{equation}\label{ineq-2}
		\abs{u(x)}^n \leq \int_\R \abs{\p_{1} u} dx_1 \prod_{i=2}^{n} \int_{\Torus} \abs{\p_i u} dx_i.
		\end{equation}
		
		\item[2) ] For any $ 1\leq i\leq n, $ and $ 1\leq q <+\infty, 1< r <+\infty, $ it holds that
		\begin{equation}\label{ineq-3}
		\int \abs{\p_i u}^p dx_i \leq C \left(\int \abs{\p_i^2 u}^r dx_i\right)^{\frac{p}{2r}} \left(\int \abs{u}^q dx_i\right)^{\frac{p}{2q}},
		\end{equation}
		where $ \frac{2}{p} = \frac{1}{r} + \frac{1}{q}, $ and $ ``\int" $ indicates the integral over $ \R $ if $ i=1 $ or $ \Torus $ if $ i>1. $
	\end{itemize}
	Thanks to \cref{ave-0-i}, it is straightforward to show \cref{ineq-2}. And \cref{ineq-3} for $ i=1 $ has been verified in \cite{Nirenberg1959}. It remains to show this inequality on $ \Torus=[0,1] $ for $ i>1. $ 
	In fact, for any bounded interval $ I\subset \R, $ it follows from \cite[(2.7)]{Nirenberg1959} that
	\begin{equation}\label{ineq-15}
	\int_I \abs{\p_i u}^p dx_i  \leq C \abs{I}^{1+p-\frac{p}{r}} \left( \int_I \abs{\p_i^2 u}^r  dx_i \right)^{\frac{p}{r}} + C \abs{I}^{-\left( 1+p-\frac{p}{r} \right)} \left( \int_I \abs{u}^q dx_i \right)^{\frac{p}{q}},
	\end{equation}
	where $ C>0 $ is independent of $ I. $ Since $ u $ is periodic in the $ x_i $ direction with period $ 1, $ the first term on the RHS of \cref{ineq-15} must be greater than the second one, as long as the length of $ I $ is equal to a large integer $ N. $
	Similar to \cite[(2.8)]{Nirenberg1959}, one can cover $ [0,N] $ by finite intervals $ I_j \subset [0,2N], $ at each of which, the first term on the RHS of \cref{ineq-15} (where $ I=I_j $) is either greater than or equal to the second one. 
	The proof is to repeat that of L. Nirenberg's in \cite{Nirenberg1959}. However, to make this paper complete, we still give the details here.
	First let $ k > 0 $ be large enough and fixed, and consider $ I=(0, \frac{N}{k}) $ in \cref{ineq-15}. If the first term on the RHS is greater than the second one, set $ I_1 = (0, \frac{N}{k}), $ and it holds that
	\begin{equation}\label{ineq-16}
	\int_{I_1} \abs{\p_i u}^p dx_i \leq C \left(\frac{N}{k}\right)^{1+p-\frac{p}{r}} \left( \int_{I_1} \abs{\p_i^2 u}^r  dx_i \right)^{\frac{p}{r}},
	\end{equation}
	where the constant $ C>0 $ is independent of either $ k $ or $ N. $
	Otherwise, if the second term on the RHS of \cref{ineq-15} is greater, then due to the choice of $ N, $ one can extend the interval $ (0, \frac{N}{k}) $ to a larger one, $ I_1 \subset [0,2N], $ until the two terms become equal. Then one has that
	\begin{equation}\label{ineq-17}
	\int_{I_1} \abs{\p_i u}^p dx_i \leq C \left(\int_{I_1} \abs{\p_i^2 u}^r dx_i\right)^{\frac{p}{2r}} \left(\int_{I_1} \abs{u}^q dx_i\right)^{\frac{p}{2q}},
	\end{equation}
	where the constant $ C>0 $ is independent of either $ k $ or $ N. $
	Starting at the end point of $ I_1, $ one can repeat this process to set $ I_2, I_3, \cdots, $ until $ [0,N] $ is covered. For the fixed $ k, $ there must be finitely many such intervals $ I_j, $ each of which is contained in $ [0,2N]. $ Then summing the estimates \cref{ineq-16} and \cref{ineq-17} together yields that
	\begin{align*}
	\int_0^N \abs{\p_i u}^p dx_i \leq~& C k \left(\frac{N}{k}\right)^{1+p-\frac{p}{r}} \left( \int_0^{N} \abs{\p_i^2 u}^r  dx_i \right)^{\frac{p}{r}} \\
	& + C \left(\int_0^{2N} \abs{\p_i^2 u}^r dx_i\right)^{\frac{p}{2r}} \left(\int_0^{2N} \abs{u}^q dx_i\right)^{\frac{p}{2q}}.
	\end{align*}
	Letting $ k\to +\infty, $ one has that
	\begin{equation*}
	\int_0^N \abs{\p_i u}^p dx_i \leq C \left(\int_0^{2N} \abs{\p_i^2 u}^r dx_i\right)^{\frac{p}{2r}} \left(\int_0^{2N} \abs{u}^q dx_i\right)^{\frac{p}{2q}},
	\end{equation*}
	which yields \cref{ineq-3}, since $ u $ is periodic in the $ x_i $ direction and $ \frac{p}{2r} + \frac{p}{2q} =1. $ 
\end{proof}

\begin{Rem}
\cref{Lem-inter} means that the GN inequality \eqref{gn} holds under the additional condition \eqref{ave-0-i}. Nevertheless, it is difficult to verify \eqref{ave-0-i} for the perturbation $u(x,t)-\ut(x,t)$ for any $x_i$ direction ($i=2,\cdots,n$) even though it holds initially, i.e., $\int_{\mathbb T} w_0(x)dx_i=0 ~ \forall i=2,\cdots,n$, which is much stronger than the original condition \eqref{ave-0}. 

\end{Rem}
	
To prove \cref{Thm}, we shall establish a  GN type inequality on the domain $\R \times \mathbb T^{n-1}$ without any additional condition by applying \cref{Lem-inter} and the following decomposition lemma. 

For any $ k\geq 1, $ set
\begin{equation}\label{A-k}
\begin{aligned}
\A_{k} = \Big\{ & u(x) \text{ is measurable on } \R\times \Torus^k \text{ with } x_1\in\R, 
\text{ and periodic in } x_i\in\Torus, \\
& \text{ satisfying } \int_\Torus u(x) d x_i \equiv 0 \text{ for all } i = 2,\cdots, k+1 \Big\}.
\end{aligned}
\end{equation}
Then for any function $u\in \A_k$, the $ k+1 $-dimensional GN inequality \eqref{gn} holds true. We are ready to decompose general function $u$ in terms of $\A_k$, i.e., 

\begin{Lem}[\bf{Decomposition Lemma}]\label{Lem-decompose}
	One can decompose $ u(x) $ as follows,
	\begin{equation}\label{dec}
	u(x) = \sum\limits_{k=0}^{n-1} u^{(k)}(x) \quad a.e.~ x\in \Omega,
	\end{equation}
	where 	
	\begin{align}
	u^{(0)}(x) & = u^{(0)}(x_1) = \int_{\Torus^{n-1}} u(x_1,x_2,\cdots,x_n) dx_2 \cdots dx_n, \label{u-0} \\
	u^{(k)}(x) & = \sum\limits_{2\leq i_1<\cdots<i_k\leq n} u_{i_1,\cdots,i_k}(x_1,x_{i_1},\cdots, x_{i_k}), \quad k=1,\cdots,n-1, \notag
	\end{align}
	where each $ u_{i_1,\cdots,i_k}(x_1,x_{i_1},\cdots, x_{i_k}) \in \A_k. $
	Moreover, for any $ m\geq 0 $ and $ 1\leq p \leq +\infty, $ it holds that
	\begin{align}
	\norm{\nabla_x^m u^{(0)}}_{L^p(\Omega)} + \sum\limits_{k=1}^{n-1} \sum\limits_{2\leq i_1<\cdots<i_k\leq n} \norm{\nabla^m_x u_{i_1,\cdots,i_k}}_{L^p(\Omega)} \leq C \norm{\nabla^m_x u}_{L^p(\Omega)},  \label{equiv}
	\end{align}
	where $ C>0 $ is independent of $ u. $
\end{Lem}

\begin{Rem}
	To avoid excessive words, we omit the assumptions in \cref{Lem-decompose} that the integrals \cref{u-0} and $$ \int_{\Torus^{n-1-k}} u(x) \frac{dx_2 \cdots dx_n}{dx_{i_1} \cdots dx_{i_k}}, \quad 1 \leq k \leq n-1, \quad 2\leq i_1<\cdots<i_k\leq n, $$ should exist, here and hereafter we use $ \int\limits_{\Torus^{n-1-k}} (\cdot) \frac{dx_2 \cdots dx_n}{dx_{i_1} \cdots dx_{i_k}} $ to denote the integral that is integrated with respect to $ x_2, \cdots,x_n $ except for $ x_{i_1}, \cdots, x_{i_k}. $
\end{Rem}

\begin{proof}
	Starting with $ u^{(0)} $ defined in \cref{u-0}, set 
	\begin{equation}\label{u-1}
	u_i(x_1,x_i) = \int_{\Torus^{n-2}} \left(u-u^{(0)}\right) \frac{dx_2 \cdots dx_n}{dx_i}, \quad i =2,\cdots, n.
	\end{equation}
It is straightforward to check that  $u_i(x_1,x_i) \in \A_1 $. Define $ u^{(1)}(x) := \sum\limits_{i=2}^n u_i$.
Assume that for all $ l = 1, 2, \cdots, k-1 $ with $ k\geq 2 $ and $ 2\leq j_1<\cdots<j_l \leq n, $  all the functions $ u_{j_1,\cdots,j_l} = u_{j_1,\cdots,j_l}(x_1,x_{j_1},\cdots, x_{j_l}) \in \A_l $ and thus $ u^{(l)} := \sum\limits_{2\leq j_1<\cdots<j_l\leq n} u_{j_1,\cdots,j_l} $ have been well-defined. Then for any fixed sequence $ 2\leq i_1 < \cdots < i_k \leq n, $ set
	\begin{align}
	u_{i_1,\cdots,i_k}(x_1,x_{i_1},\cdots,x_{i_k}) &:= \int_{\Torus^{n-1-k}} \Big( u - \sum\limits_{l=0}^{k-1} u^{(l)} \Big) \frac{dx_2 \cdots dx_n}{dx_{i_1}\cdots dx_{i_k}} \label{u-k} \\
	& = \int_{\Torus^{n-1-k}} \Big( u- \sum\limits_{l=0}^{k-2}u^{(l)} - \sum\limits_{2\leq j_1<\cdots<j_{k-1}\leq n} u_{j_1,\cdots,j_{k-1}} \Big) \frac{dx_2 \cdots dx_n}{dx_{i_1}\cdots dx_{i_k}}. \notag
	\end{align}
	It is noted that for $ r=1,\cdots,k, $ one has 
	\begin{equation}
	\int_{\Torus} u_{j_1,\cdots,j_{k-1}} dx_{i_r} = \begin{cases}
	0 & \text{ if } i_r \in \{ j_1,\cdots,j_{k-1} \}, \\
	u_{j_1,\cdots,j_{k-1}} \qquad & \text{ if } i_r \notin \{ j_1,\cdots,j_{k-1} \}.
	\end{cases}
	\end{equation}
	Then it holds that 
	\begin{align*}
	\int_{\Torus}  u_{i_1,\cdots,i_k} dx_{i_r}
	=~ & \int_{\Torus^{n-k}} \Big( u - \sum\limits_{l=0}^{k-2} u^{(l)} \Big) \frac{dx_2 \cdots dx_n}{dx_{i_1}\cdots dx_{i_{r-1}}dx_{i_{r+1}} \cdots dx_{i_{k}}} \\
	& - \sum\limits_{
	2\leq j_1<\cdots<j_{k-1}\leq n} \int_{\Torus^{n-k}} u_{j_1,\cdots,j_{k-1}} \frac{dx_2 \cdots dx_n}{dx_{i_1}\cdots dx_{i_{r-1}}dx_{i_{r+1}} \cdots dx_{i_{k}}} \\
	=~& u_{i_1, \cdots,i_{r-1},i_{r+1},\cdots, i_k} - u_{i_1, \cdots,i_{r-1},i_{r+1},\cdots, i_k} \\
	=~ & 0,
	\end{align*}	
	which means $ u_{i_1,\cdots,i_k} \in \A_k$ and then the decomposition  \eqref{dec} for function $u$ holds. 
	It remains to show \cref{equiv}. It follows from \cref{u-0} and Minkowski inequality that
	\begin{equation*}
	\begin{array}{ll}
	\norm{\nabla_x^m u^{(0)}}_{L^p(\R)} & \leq \norm{ \norm{\nabla_x^m u}_{L^1(\Torus^{n-1};dx_2\cdots dx_n)} }_{L^p(\R;dx_1)} \\
	& \leq \norm{ \norm{\nabla_x^m u}_{L^p(\R;dx_1)} }_{L^1(\Torus^{n-1};dx_2\cdots dx_n)} \\
	& \leq \norm{\nabla^m_x u}_{L^p(\Omega)},
	\end{array}
	\end{equation*} 
	where the last inequality is derived from the H\"{o}lder inequality. Similarly, one can obtain from \cref{u-1} that for any $ 2\leq i\leq n, $
	\begin{equation*}
	\norm{\nabla_x^m u_i} \leq \norm{\nabla_x^m \left(u-u^{(0)}\right)}_{L^p(\Omega)} \leq 2 \norm{\nabla^m_x u}_{L^p(\Omega)}.
	\end{equation*}
	The remaining functions $ u_{i_1,\cdots,i_k} $ defined in \cref{u-k} can be proved in the same way, and then the proof of Lemma \ref{Lem-decompose} is completed.
\end{proof}
    
\begin{proof}[Proof of \cref{Thm-G-N}]
We first decompose $ u(x) = \sum\limits_{k=0}^{n-1} u^{(k)}(x) $ as in \cref{Lem-decompose}. Then it follows from Theorem \ref{Lem-inter} and Lemma \ref{Lem-decompose} that $ u^{(0)} \in L^q(\R), \nabla^m u^{(0)} \in L^r(\R), $ and it satisfies the $ 1 $-d GN inequality \eqref{gn}; each $ u_{i_1,\cdots,i_k} \in L^q(\Omega), \nabla^m u_{i_1,\cdots,i_k} \in L^r(\Omega) $ and it satisfies the $ (k+1) $-d GN inequality \eqref{gn}. Hence, 
\begin{align*}
	\norm{\nabla^j u^{(0)}}_{L^p(\R)} \leq~& C \norm{\nabla^m u^{(0)}}_{L^{r}(\R)}^{\theta_0} \norm{u^{(0)}}_{L^{q}(\R)}^{1-\theta_0}
	\leq C \norm{\nabla^m u}_{L^{r}(\Omega)}^{\theta_0} \norm{u}_{L^{q}(\Omega)}^{1-\theta_0}; \\
	\norm{\nabla^j u^{(k)}}_{L^p(\Omega)} \leq~& \sum\limits_{2\leq i_1<\cdots<i_k\leq n} \norm{\nabla^j u_{i_1,\cdots,i_k} }_{L^p(\Omega)} \\
     \leq~& C \sum\limits_{2\leq i_1<\cdots<i_k\leq n} \norm{\nabla^m u_{i_1,\cdots,i_k}}_{L^{r}(\Omega)}^{\theta_k} \norm{u_{i_1,\cdots,i_k}}_{L^{q}(\Omega)}^{1-\theta_k} \\
	\leq~& C \norm{\nabla^{m} u}_{L^{r}(\Omega)}^{\theta_k} \norm{u}_{L^{q}(\Omega)}^{1-\theta_k}, \qquad 1\leq k\leq n-1,
\end{align*}
where the indices $ \theta_0,\cdots, \theta_{n-1} $ are that introduced in \cref{Thm-G-N}. The proof is finished.

\end{proof}

\vspace{0.2cm}

In order to use the $ L^p $ energy method developed in \cite{KNM2004}, we shall establish the following interpolation inequality by \cref{Thm-G-N}.

\begin{Cor}[Interpolation Inequality in $\Omega$]\label{Cor-inter}
	For any $ 2\leq p<\infty $ and $ 1\leq q \leq p, $ it holds that
	\begin{align}\label{inter-in}
	\norm{u}_{L^p(\Omega)}\leq C \sum\limits_{k=0}^{n-1} 
	\norm{\nabla(\abs{u}^{\frac{p}{2}})}^{\frac{2\gamma_k}{1+\gamma_k p}}_{L^2(\Omega)}\norm{u}^{\frac{1}{1+\gamma_k p}}_{L^q(\Omega)},
	\end{align}
	where $\gamma_k= \frac{k+1}{2} \left(\frac{1}{q}- \frac{1}{p}\right) $ and the constant $C=C(p,q,n)>0$ is independent of $ u. $
\end{Cor}

\begin{proof}
For the function $ v(x) = \abs{u(x)}^{\frac{p}{2}}, $ it follows from \cref{Thm-G-N} that
\begin{align*}
\norm{v}_{L^2(\Omega)} \leq  C \sum\limits_{k=0}^{n-1} \norm{\nabla v}_{L^2(\Omega)}^{\theta_k} \norm{v}_{L^{2 r_k}(\Omega)}^{1-\theta_k},
\end{align*}
where $ \frac{1}{2} = \left( \frac{1}{2} - \frac{1}{k+1} \right) \theta_k + \frac{1}{2 r_k} \left(1-\theta_k\right) $ with $ \frac{1}{2} \leq r_k \leq 1 $ and $ 0 \leq \theta_k \leq 1 $ for $ k=0,1,\cdots,n-1. $ This yields that
\begin{equation}\label{ineq-4}
\norm{u}_{L^p(\Omega)} \leq C \sum\limits_{k=0}^{n-1} \norm{\nabla \abs{u}^{\frac{p}{2}}}_{L^2(\Omega)}^{\frac{2}{p}\theta_k} \norm{u}_{L^{pr_k}(\Omega)}^{1-\theta_k}.
\end{equation}
If $ 2\leq p\leq 2q, $ choosing $ r_k = \frac{q}{p} \geq \frac{1}{2} $ for $ k=0,\cdots,n-1 $ can finish the proof. 
If $ p>2q, $ let $ \frac{1}{2}<r_k <1, $ which will be determined later. Since $ q < p r_k < p, $  it follows from interpolation  that
\begin{equation}\label{ineq-5}
\norm{u}_{L^{p r_k}(\Omega)} \leq \norm{u}_{L^q(\Omega)}^{1-\rho_k} \norm{u}_{L^p(\Omega)}^{\rho_k},
\end{equation}
where $ \rho_k = \frac{p r_k - q}{(p-q) r_k} \in (0,1). $ Plugging \cref{ineq-5} into \cref{ineq-4} yields that
\begin{equation}\label{ineq-6}
\norm{u}_{L^p(\Omega)} \leq C \sum\limits_{k=0}^{n-1} \norm{\nabla \abs{u}^{\frac{p}{2}}}_{L^2(\Omega)}^{\frac{2}{p}\theta_k} \norm{u}_{L^q(\Omega)}^{(1-\rho_k)(1-\theta_k)} \norm{u}_{L^p(\Omega)}^{\rho_k(1-\theta_k)}.
\end{equation}
For $ k = 0,1,\cdots, n-1 $ with $ n\geq 2 $ and $ r_k \in \left( \frac{1}{2},1\right), $ one can obtain by simple calculations that
\begin{equation}\label{ineq-7}
\rho_k(1-\theta_k) = \frac{p}{p-q} \times \frac{r_k-\frac{q}{p}}{\frac{k+1}{2} -\frac{k-1}{2} r_k} \quad \in \left( \frac{2(p-2q)}{(k+3)(p-q)},1 \right).
\end{equation}
Thus, for any $ \delta \in \left(\frac{2(p-2q)}{3(p-q)},1\right), $ there exist $ r_0, r_1, \cdots, r_{n-1} \in \left(\frac{1}{2}, 1\right) $ such that $ \rho_k(1-\theta_k) = \delta $ for any $ k=0,1,\cdots,n-1. $ Then \cref{ineq-6} yields that
\begin{equation*}
\norm{u}_{L^p(\Omega)} \leq C \sum\limits_{k=0}^{n-1} \norm{\nabla \abs{u}^{\frac{p}{2}}}_{L^2(\Omega)}^{\frac{2\theta_k}{p(1-\delta)}} \norm{u}_{L^q(\Omega)}^{\frac{(1-\rho_k)(1-\theta_k)}{1-\delta}}.
\end{equation*}
A direct calculation implies that $ \frac{2\theta_k}{p(1-\delta)} = \frac{2\gamma_k}{1+\gamma_k p} $ and $ \frac{(1-\rho_k)(1-\theta_k)}{1-\delta} = \frac{1}{1+\gamma_k p} $ for all $ k=0,\cdots,n-1, $ where $ \gamma_k = \frac{k+1}{2} \left(\frac{1}{q}-\frac{1}{p}\right)$, and  the proof is completed.
\end{proof}

\begin{Lem}[\cite{KNM2004} Lemma 2.2]\label{Lem-II-der}
	For any $ 2\leq p<\infty, $ it holds that
	\begin{equation}\label{inter-deriv}
	\norm{\p_i u}_{L^p(\Omega)}\leq C
	\norm{\p_i \left(\abs{\p_i u}^{\frac{p}{2}}\right) }^{\frac{2}{p+2}}_{L^2(\Omega)}\norm{u}^{\frac{2}{p+2}}_{L^p(\Omega)}, \quad i=1, \cdots, n,
	\end{equation}
	where the constant $C>0$ is independent of $ u. $
\end{Lem}
\begin{proof}
The inequality \eqref{inter-deriv} has been established in \cite{KNM2004} for the whole space $ \R^n, $ and it is still true in the domain $\Omega$ with the aid of integration by parts and the H\"{o}lder inequality as in \cite{KNM2004}.
\end{proof}


\section{A priori estimates and proof}\label{Sec-Apriori}

Denote the perturbation by $ \phi(x,t):= u(x,t)-\ut(x,t), $ which is periodic in the $ x_i $ direction for $ i=2,\cdots,n, $ and satisfies
\begin{align}
\p_t \phi +\sum_{i=1}^n\partial_i\left[f_i(\ut+\phi)-f_i(\ut)\right]  & = \triangle \phi-h, \label{equ-phi} \\
\phi(x,0) & = 0.  \label{ic-phi}
\end{align}
\begin{Rem}
	In fact, given the initial data $ u_0(x) = \ut^R(x_1,0) + v_0(x_1) +w_0(x) $ stated in \cref{Rem-ic}, the initial condition \cref{ic-phi} turns to $ \phi(x,0) = v_0(x_1) \in L^1(\R) \cap H^1(\R) $ instead, which makes no difference in the following proof.
\end{Rem}

We shall prove the global existence and large time behavior of solution $\phi(x,t)$ to the Cauchy problem \eqref{equ-phi} and \eqref{ic-phi}. The global existence  can be established by obtaining the a priori estimates \cref{aprio-1} and \cref{aprio-2} below, since the local existence of the solution $ \phi $ to \cref{equ-phi} with the initial data in $ L^1(\Omega) \cap H^1(\Omega) $ is standard. 

\begin{Thm}[A priori estimates]\label{Thm-apriori}
	Assume that $ \phi(x,t)$ is the unique smooth solution  to \cref{equ-phi},\cref{ic-phi} for any $t\in [0,T]$, then it holds that 
	\begin{align}
	\norm{\phi(t)}_{L^p(\Omega)}
	& \leq C_p (1+t)^{-\frac{1}{2}+\frac{1}{2p}} \qquad \forall p\in[1,+\infty), \label{aprio-1} \\
	\norm{\nabla \phi(t)}_{L^p(\Omega)}
	& \leq C_p (1+t)^{-1+\frac{1}{2p}}, \qquad \forall p\in[2,+\infty), \label{aprio-2}
	\end{align}
	where the constants $ C>0 $ and $C_p>0$ are independent of $t$.
\end{Thm}

\begin{proof}
	The uniform bound of $ \norm{\phi(t)}_{L^\infty(\Omega)} $ follows from the maximum principle easily. And following the $L^p$ energy method as in \cite{KNM2004}, we first prove \eqref{aprio-1}.
	
	\textbf{Step 1. } We first show the $ L^1 $ estimates. Given $ \delta>0, $ let $ S_\delta(\eta) $ be a $ C^2 $ convex approximation to the function $ \abs{\eta}, $ e.g.
	\begin{equation*}
	S_{\delta}(\eta) =
	\begin{cases}
	-\eta, \quad &\eta \leq -\delta; \\
	-\frac{\eta^4}{8\delta^3} + \frac{3\eta^2}{4\delta} + \frac{3\delta}{8}, \quad &-\delta<\eta\leq \delta; \\
	\eta, \quad &\eta > \delta.
	\end{cases}
	\end{equation*}
	Multiplying $ S_\delta'(\phi) $ on both sides of \cref{equ-phi} yields that
	\begin{equation}\label{ineq-11}
	\begin{aligned}
	& \p_t S_\delta(\phi) + S_\delta''(\phi) \abs{\nabla \phi}^2 + \int_0^\phi S_\delta''(\eta) \left( f_1'(\ut+\eta)-f_1'(\ut) \right) d\eta~ \p_1 \ut^R \\
	=~ & \sum\limits_{i=1}^n \p_i \{\cdots\} - \int_0^\phi S_\delta''(\eta) \left( f_1'(\ut+\eta)-f_1'(\ut) \right) d\eta~ \p_1 \left(\ut - \ut^R\right) \\
	& -  \sum\limits_{i=2}^n \int_0^\phi S_\delta''(\eta) \left( f_i'(\ut+\eta)-f_i'(\ut) \right) d\eta~ \p_i \ut - S_\delta'(\phi) h,
	\end{aligned}	
	\end{equation}
	where
	\begin{align*}
	\{\cdots\} = S_\delta'(\phi) \p_i \phi - S_\delta'(\phi) \left(f_i(\ut+\phi)-f_i(\ut)\right) + \int_0^\phi S_\delta''(\eta) \left( f_i(\ut+\eta)-f_i(\ut) \right) d\eta.
	\end{align*}
	Since $ S_\delta''\geq 0, f_1''>0, \p_1 \ut^R>0 $ and $ \abs{\phi} \leq S_\delta(\phi), $  integrating \cref{ineq-11} over $ \Omega $, together with Lemmas \ref{Lem-per} and \ref{Lem-h}, gives that
	\begin{align}
	\frac{d}{dt} \int_\Omega S_\delta(\phi) dx \leq~& C e^{-\alpha t} \int_\Omega \sum\limits_{i=1}^n \abs{ \int_0^\phi S_\delta''(\eta) \left( f_i'(\ut+\eta)-f_i'(\ut) \right) d\eta } dx + C \norm{h}_{L^1(\Omega)} \notag \\
	\leq~& C e^{-\alpha t} \int_\Omega \abs{ \int_0^\phi S_\delta''(\eta) \abs{\eta} d\eta } dx + C e^{-\alpha t} \notag \\
	\leq~& C e^{-\alpha t} \norm{\phi}_{L^1(\Omega)} + C e^{-\alpha t} \notag \\
	\leq~& C e^{-\alpha t} \int_\Omega S_\delta(\phi) dx + C e^{-\alpha t}, \notag
	\end{align}
	where $ C>0 $ is independent of $ \delta. $ Then by the Gronwall inequality and letting $ \delta \rightarrow 0+, $  one has that
	\begin{equation}\label{L-1}
	\norm{\phi(t)}_{L^1(\Omega)}  \leq C.
	\end{equation}
	For $ p \in [2,+\infty), $ multiplying \cref{equ-phi} by $\abs{\phi}^{p-2}\phi$ yields that
	\begin{equation}\label{lp}
	\begin{aligned}
	& \frac{1}{p} \p_t \abs{\phi}^{p} + \frac{4(p-1)}{p^2} \abs{\nabla \abs{\phi}^{\frac{p}{2}}}^2
	+ (p-1) \int_0^\phi \left(f'_1(\ut +\eta)-f'_1(\ut)\right) |\eta|^{p-2} \mathrm{d} \eta ~\p_1 \ut^R \\
	=~ & \sum\limits_{i=1}^n \p_i \{\cdots\} - (p-1) \int_{0}^{\phi}(f'_1(\ut +\eta)-f'_1(\ut ))|\eta|^{p-2} \mathrm{d} \eta ~\p_1 \left(\ut-\ut^R\right) \\
	& - (p-1) \sum\limits_{i=2}^n \int_{0}^{\phi}(f'_i(\ut +\eta)-f'_i(\ut ))|\eta|^{p-2} \mathrm{d} \eta ~\p_i \ut - \abs{\phi}^{p-2} \phi h,
	\end{aligned}	
	\end{equation}
	where
	\begin{align*}
	\{\cdots\} = ~& \abs{\phi}^{p-2} \phi \p_i \phi - (f_i(\ut +\phi)-f_i(\ut )) \abs{\phi}^{p-2} \phi \\
	& + (p-1) \int_{0}^{\phi} \left(f_i(\ut +\eta)-f_i(\ut )\right) |\eta|^{p-2} \mathrm{d} \eta.
	\end{align*}
	Since
	\begin{equation*}
	(p-1) \int_0^\phi (f'_1(\ut +\eta)-f'_1(\ut)) |\eta|^{p-2} \mathrm{d} \eta ~\p_1 \ut^R \geq C \p_1 \ut^R \abs{\phi}^p \geq 0,
	\end{equation*}
	then integrating \cref{lp} over $\Omega$, together with Lemma \ref{Lem-h}, gives that
	\begin{align}
	& \frac{d}{dt} \norm{\phi}^p_{L^p(\Omega)} + \norm{\nabla \abs{\phi}^{\frac{p}{2}}}^2_{L^2(\Omega)} + \norm{\left(\p_1 \ut^R\right)^{\frac{1}{p}} \phi }_{L^p(\Omega)}^p \notag \\
	\leq~ & C e^{-\alpha t} \norm{\phi}^p_{L^p(\Omega)} + C \left(\norm{\phi}^p_{L^p(\Omega)} \right)^{1-\frac{1}{p}}  \norm{h}_{L^p(\Omega)} \notag\\
	\leq~ & C e^{-\alpha t} \norm{\phi}^p_{L^p(\Omega)} + C e^{-\alpha t}. \label{lpe}
	\end{align}
	Multiplying \cref{lpe} by $(1+t)^\beta $, where $ \beta > \frac{n}{2} \left(p-1\right) $ is a constant, then integrating the resulting equation over $(0,T)$ yields that
	\begin{align}
	& (1+T)^\beta  \norm{\phi(T)}^p_{L^p(\Omega)}  + \int_0^T(1+t)^\beta \norm{\nabla \abs{\phi}^{\frac{p}{2}}}^2_{L^2(\Omega)} dt \notag \\
	\leq~ & C  \int_{0}^{T} e^{-\alpha t} (1+t)^{\beta} \norm{\phi}^p_{L^p(\Omega)} dt + C +  C  \int_{0}^{T}(1+t)^{\beta-1} \norm{\phi}^p_{L^p(\Omega)} dt \notag \\
	\leq~ & C +  C  \int_{0}^{T}(1+t)^{\beta-1} \norm{\phi}^p_{L^p(\Omega)} dt. \label{lpee} 
	\end{align} 
	It follows from \cref{Cor-inter} and \cref{L-1} that
	\begin{align}
	& C \int_0^T (1+t)^{\beta-1} \norm{\phi}^p_{L^p(\Omega)} dt \notag \\
	\leq~& C \sum\limits_{k=0}^{n-1} \int_0^T (1+t)^{\beta-1} \norm{\nabla \abs{\phi}^{\frac{p}{2}}}^{\frac{2\gamma_k p}{1+\gamma_k p}}_{L^2(\Omega)} \norm{\phi}^{\frac{p}{1+\gamma_k p}}_{L^1(\Omega)} dt \notag \\
	\leq~& \frac{1}{2} \int_0^T (1+t)^\beta \norm{\nabla \abs{\phi}^{\frac{p}{2}}}^2_{L^2(\Omega)} dt + C \left(1+T\right)^{\beta- \gamma_0 p}, \label{ineq-1}
	\end{align}
	where $ \gamma_k = \frac{k+1}{2}\left(1-\frac{1}{p}\right).$ This, together with \cref{lpee}, yields that
	\begin{align*}
	(1+T)^\beta \norm{\phi(T)}^p_{L^p(\Omega)}  + \int_0^T (1+t)^\beta  \norm{\nabla \abs{\phi}^{\frac{p}{2}}}^2_{L^2(\Omega)} dt
	\leq C + C \left(1+T\right)^{\beta-\frac{p-1}{2}}.
	\end{align*} 
	Thus \cref{aprio-1} holds for $ p\in[2,+\infty). $	
	The case for $ p\in(1,2) $ follows from the interpolation.
	
	\vspace{0.2cm}
	
	\textbf{Step 2. } We now prove \cref{aprio-2}.
	Let $ \psi_i := \p_i \phi. $ Taking the derivative on \cref{equ-phi} with respect to $x_i$ yields that
	\begin{equation*}
	\p_t \psi_i + \sum_{j=1}^{n}\partial_j \left( f'_j(\ut )\psi_i\right) + \sum_{j=1}^{n} \partial_{j} \left[\left(f'_j(\ut +\phi)-f'_j(\ut )\right) \left(\p_i \ut + \psi_i\right) \right] = \triangle \psi_{i}-\p_i h.	
	\end{equation*}
	Multiplying the result by $ \abs{\psi_i}^{p-2} \psi_i, $ we arrive at
	\begin{equation}\label{so}
	\begin{aligned}
	&\frac{1}{p} \p_t \abs{\psi_i}^p + \frac{4(p-1)}{p^2} \abs{\nabla \abs{\psi_i}^{\frac{p}{2}} }^2 + \frac{p-1}{p} f_1''(\ut) \p_1 \ut^R \abs{\psi_i}^p \\
	=~ & \sum_{j=1}^n \p_j \{\cdots\} - \frac{p-1}{p} f_1''(\ut) \p_1 \left(\ut-\ut^R\right) \abs{\psi_i}^p - \frac{p-1}{p} \sum_{j=2}^n f_j''(\ut ) \p_j \ut \abs{\psi_i}^p \\
	& + (p-1) \sum_{j=1}^n \left(f'_j(\ut +\phi)-f'_j(\ut )\right) \left(\p_i \ut + \psi_i\right) \abs{\psi_i}^{p-2} \p_j \psi_i \\
	& - \p_i \left( \abs{\psi_i}^{p-2} \psi_i h \right) + (p-1) \abs{\psi_i}^{p-2} \p_i\psi_i h,	
	\end{aligned}	
	\end{equation}
	where
	\begin{align*}
	\{\cdots\} = \abs{\psi_i}^{p-2} \psi_{i}\partial_{j}\psi_i - \frac{1}{p} f_j'(\ut ) \abs{\psi_i}^p  - \left(f'_j(\ut +\phi)-f'_j(\ut )\right) \left(\p_i \ut + \psi_i\right) \abs{\psi_i}^{p-2} \psi_i.
	\end{align*}
	Then integrating (\ref{so}) over $\Omega$ and using \cref{Lem-h} yield  that
	\begin{equation}\label{ineq-8}
	\begin{aligned}
	& \frac{d}{dt} \norm{\psi_{i}}_{L^p(\Omega)}^p + \norm{\nabla \abs{\psi_i}^{\frac{p}{2}}}_{L^2(\Omega)}^2 + \norm{\left(\p_1 \ut^R\right)^{\frac{1}{p}} \psi_i}^p_{L^p(\Omega)} \\
	\leq~ & C e^{-\alpha t} \norm{\psi_{i}}_{L^p(\Omega)}^p + \underbrace{ C \int_{\Omega} \abs{\p_i \ut} \abs{\phi} \abs{\psi_i}^{p-2} \abs{\nabla\psi_{i}} dx}_{I_1} \\
	& + \underbrace{ C \int_{\Omega}  \abs{\phi} \abs{\psi_i}^{p-1} \abs{\nabla\psi_{i}} dx}_{I_2} + \underbrace{ C \int_\Omega \abs{\psi_i}^{p-2} \abs{\nabla \psi_i} \abs{h} dx }_{I_3}.
	\end{aligned}
	\end{equation}
	First, it follows from Lemma \ref{Lem-h} that
	\begin{align}
	I_3 & = C \int_\Omega \abs{ \nabla \abs{\psi_i}^{\frac{p}{2}}} \left( e^{-\alpha t} \abs{\psi_i}^p \right)^{\frac{p-2}{2p}} e^{\frac{p-2}{2p} \alpha t} \abs{h} dx  \notag \\
	& \leq \frac{1}{8} \norm{\nabla \abs{\psi_i}^{\frac{p}{2}}}_{L^2(\Omega)}^2 + C e^{-\alpha t} \norm{\psi_i}_{L^p(\Omega)}^p + C e^{-\alpha \left( \frac{p}{2}+1 \right) t}, \label{I-3}
	\end{align}
	where the H\"{o}lder inequality for $ \frac{1}{2} + \frac{p-2}{2p} + \frac{1}{p} = 1 $ is used. Following \cite{KNM2004}, one can claim that
	\begin{equation}\label{claim}
	\int_{\Omega} \abs{\phi} \abs{\psi_i}^{\frac{p}{2}} \abs{\nabla \abs{\psi_{i}}^\frac{p}{2}} dx \leq C \sum\limits_{k=0}^{n-1} \norm{\phi}_{L^{2(k+1)}(\Omega)} \norm{\psi_i}_{L^p(\Omega)}^{\frac{p}{4}} \norm{ \nabla \abs{\psi_i}^{\frac{p}{2}} }^{\frac{3}{2}}_{L^2(\Omega)}.
	\end{equation}	
	In fact, we first decompose $ v := \abs{\psi_i}^{\frac{p}{2}} = \sum\limits_{k=0}^{n-1} v^{(k)} $ as in \cref{Thm-G-N}, then 
	\begin{align*}
	& \int_{\Omega} \abs{\phi} \abs{\psi_i}^{\frac{p}{2}} \abs{\nabla \abs{\psi_{i}}^\frac{p}{2}} dx \leq \sum\limits_{k=0}^{n-1} \int_{\Omega} \abs{\phi} \abs{v^{(k)}} \abs{\nabla v} dx  \\
	\leq~& \norm{\phi}_{L^2(\Omega)} \norm{v^{(0)}}_{L^\infty(\R)} \norm{\nabla v}_{L^2(\Omega)} + \norm{\phi}_{L^4(\Omega)} \norm{v^{(1)}}_{L^4(\Omega)} \norm{\nabla v}_{L^2(\Omega)}  \\
	& + \sum\limits_{k=2}^{n-1} \norm{\phi}_{L^{2(k+1)}(\Omega)} \norm{\abs{v^{(k)}}^{\frac{1}{2}}}_{L^4(\Omega)} \norm{\abs{v^{(k)}}^{\frac{1}{2}}}_{L^{\frac{4(k+1)}{k-1}}(\Omega)}
	\norm{\nabla v}_{L^2(\Omega)}  \\
	\leq~& \sum\limits_{k=0}^{n-1} 
	\norm{\phi}_{L^{2(k+1)}(\Omega)} \norm{v}_{L^2(\Omega)}^{\frac{1}{2}} \norm{\nabla v}^{\frac{3}{2}}_{L^2(\Omega)},
	\end{align*}
	which yields \cref{claim}. Then combining \cref{aprio-1} and \cref{claim}, the term $ I_2 $ in \cref{ineq-8} satisfies that
	\begin{align}
	I_2 
	& \leq C (1+t)^{-\frac{1}{4}} \norm{\psi_i}_{L^p(\Omega)}^{\frac{p}{4}} \norm{ \nabla \abs{\psi_i}^{\frac{p}{2}} }^{\frac{3}{2}}_{L^2(\Omega)} \notag \\
	& \leq C (1+t)^{-1} \norm{\psi_i}_{L^p(\Omega)}^{p} + \frac{1}{8} \norm{ \nabla \abs{\psi_i}^{\frac{p}{2}} }^{2}_{L^2(\Omega)}. \label{ineq-I2}
	\end{align}
	For $ I_1, $ if $ i \neq 1, $ it follows from the H\"{o}lder inequality and \cref{aprio-1} that
	\begin{equation}\label{ineq-9}
	\begin{aligned}
	I_1 & \leq C e^{-\alpha t} \int_\Omega \abs{\phi} \abs{\psi_i}^{\frac{p-2}{2}}\abs{\psi_i}^{\frac{p-2}{2}} \abs{\nabla \psi_i} dx  \\
	& \leq C e^{-\alpha t} \norm{\phi}_{L^p(\Omega)} \norm{\psi_i}_{L^p(\Omega)} \norm{\nabla \abs{\psi_i  }^{\frac{p}{2}} }_{L^2(\Omega)} \\
	& \leq C e^{-\alpha t} + C e^{-\alpha t} \norm{\psi_i}_{L^p(\Omega)}^p + \frac{1}{8} \norm{\nabla \abs{\psi_i  }^{\frac{p}{2}}}_{L^2(\Omega)}^2.
	\end{aligned}
	\end{equation}
	If $ i=1, $ then
	\begin{equation}\label{ineq-10}
	\begin{aligned}
	I_1 \leq~ & C \int_\Omega \p_1 \ut^R \abs{\phi} \abs{\psi_1}^{p-2} \abs{\nabla \psi_1} dx + C e^{-\alpha t} \int_\Omega \abs{\phi} \abs{\psi_1}^{p-2} \abs{\nabla \psi_1} dx,
	\end{aligned}
	\end{equation}
	where the second term on the RHS of \cref{ineq-10} can be estimated in the same way as in \cref{ineq-9}. For the first term, one can use \cref{Lem-II-der} to obtain that
	\begin{align}
	\int_\Omega \p_1 \ut^R \abs{\phi} \abs{\psi_1}^{p-2} \abs{\nabla \psi_1} dx & \leq C \norm{\p_1 \ut^R}_{L^\infty(\Omega)} \norm{\phi}_{L^p(\Omega)}^{\frac{2p}{p+2}} \norm{\nabla \abs{\psi_1}^{\frac{p}{2}}}_{L^2(\Omega)}^{\frac{2p}{p+2}} \notag \\
	&  \leq C \norm{\p_1 \ut^R}_{L^\infty(\Omega)}^{\frac{p+2}{2}} \norm{\phi}_{L^p(\Omega)}^{p} + \frac{1}{8} \norm{\nabla \abs{\psi_1}^{\frac{p}{2}}}_{L^2(\Omega)}^{2}. \label{ineq-12}
	\end{align}
	Thus, collecting \cref{I-3,ineq-I2,ineq-9,ineq-10,ineq-12} and applying \cref{entropy}, one has that
	\begin{equation}\label{ineq-13}
	\frac{d}{dt} \norm{\psi_{i}}_{L^p(\Omega)}^p + \norm{\nabla \abs{\psi_i}^{\frac{p}{2}}}_{L^2(\Omega)}^2 
	\leq C e^{-\alpha t} + C (1+t)^{-1} \norm{\psi_i}_{L^p(\Omega)}^{p} +  C (1+t)^{-\frac{p+2}{2}} \norm{\phi}_{L^p(\Omega)}^{p}.
	\end{equation}
	Multiplying \cref{ineq-13} by $ (1+t)^\beta $ with $\beta>p+1$ and then integrating the result over $ [0,T], $ one has that
	\begin{align}
	& (1+t)^\beta \norm{\psi_i(t)}_{L^p(\Omega)}^p + \int_0^T  (1+t)^\beta \norm{\nabla \abs{\psi_i}^{\frac{p}{2}}}_{L^2(\Omega)}^2 dt \notag \\
	\leq~ & C + C \int_0^T (1+t)^{\beta-1} \norm{\psi_i(t)}_{L^p(\Omega)}^p dt + C \int_0^T (1+t)^{\beta-\frac{p+2}{2}}  \norm{\phi}_{L^p(\Omega)}^{p} dt. \label{ineq-14}
	\end{align}
	It follows from \cref{Lem-II-der} that
	\begin{align*}
	&  \int_0^T (1+t)^{\beta-1} \norm{\psi_i(t)}_{L^p(\Omega)}^p dt \notag \\
	\leq~ & C \int_0^T (1+t)^{\beta-1} \norm{\nabla \abs{\psi_i}^\frac{p}{2}}_{L^2(\Omega)}^{\frac{2p}{p+2}} \norm{\phi}_{L^p(\Omega)}^{\frac{2p}{p+2}} dt \notag \\
	\leq~ & \frac{1}{2} \int_0^T  (1+t)^\beta \norm{\nabla \abs{\psi_i}^\frac{p}{2}}_{L^2(\Omega)}^2 dt + C \int_0^T (1+t)^{\beta-\frac{p+2}{2}}  \norm{\phi}_{L^p(\Omega)}^p dt.
	\end{align*}
	From \cref{aprio-1}, one can get that
	\begin{align*}
	\int_0^T (1+t)^{\beta-\frac{p+2}{2}} \norm{\phi(t)}_{L^p(\Omega)}^p dt & \leq C \int_0^T (1+t)^{\beta-\frac{p+2}{2}}  (1+t)^{-\frac{p-1}{2}} dt \notag \\
	& \leq C (1+T)^{\beta-p+\frac12}.
	\end{align*}
	The proof of \cref{Thm-apriori} is finished.
\end{proof}

\vspace{0.3cm}

\begin{proof}[Proof of \cref{Thm}]
	It follows from Theorems \ref{Thm-G-N} and \ref{Thm-apriori} that
	\begin{align*}
	\norm{\phi}_{L^\infty(\Omega)}
	& \leq C \sum\limits_{k=0}^{n-1} \norm{\nabla \phi}_{L^{p_k}(\Omega)}^{\theta_k} \norm{\phi}_{L^{q_k}(\Omega)}^{1-\theta_k} \notag \\
	& \leq C  \sum\limits_{k=0}^{n-1} (1+t)^{-\left[ \left(1-\frac{1}{2p_k}\right) \theta_k + \left(\frac{1}{2}-\frac{1}{2q_k}\right) (1-\theta_k) \right]},
	\end{align*}
	where $ 0 = \left(\frac{1}{p_k} - \frac{1}{k+1}\right) \theta_k + \frac{1}{q_k} (1-\theta_k), $ $ \max\{k+1,2 \} < p_k <+\infty $ and $ 1\leq q_k <+\infty $ for $ k=0,1,\cdots,n-1$. A direct calculation yields that
	\begin{align*}
	\left(1-\frac{1}{2p_k}\right) \theta_k + \left(\frac{1}{2}-\frac{1}{2q_k}\right)(1-\theta_k) = \frac{1}{2} + \frac{k\theta_k}{2(k+1)} \geq \frac{1}{2},
	\end{align*}
	which implies that
	\begin{equation*}
	\norm{\phi}_{L^\infty(\R^n)} = \norm{\phi}_{L^\infty(\Omega)} \leq C (1+t)^{-\frac{1}{2}}.
	\end{equation*}
	Thus, 
	\begin{align*}
	\norm{u-\ut^R}_{L^\infty(\R^n)} \leq \norm{\phi}_{L^\infty(\R^n)} + \norm{\ut-\ut^R}_{L^\infty(\R^n)} \leq C (1+t)^{-\frac{1}{2}}.
	\end{align*}

\end{proof}


\vspace{1.5cm}


\vspace{1cm}

\begin{thebibliography}{aa}
		\footnotesize
\bibitem{Dafermos1995} C. M. Dafermos, \emph{Large time behavior of periodic solutions of hyperbolic systems of conservation laws}, Journal of Differential Equations 121 (1995), no. 1, 183-202.		
\bibitem{Dafermos2003} C. M. Dafermos, \emph{Long time behavior of periodic solutions to scalar conservation laws in several space dimensions}, SIAM J. Math. Anal. 45 (2013), no. 4, 2064-2070.
\bibitem{GL1970} J. Glimm and P. Lax, \emph{Decay of solutions of systems of nonlinear hyperbolic conservation laws}, Memoirs of the American Mathematical Society, No. 101, American Mathematical Society, Providence, R.I., 1970.
\bibitem{Goodman1986} J. Goodman, \emph{Nonlinear asymptotic stability of viscous shock profiles for conservation laws}, Arch. Ration. Mech.
Anal. 95 (4) (1986) 325-344.
\bibitem{HH2020} L. He and F.  Huang, Nonlinear stability of large amplitude viscous shock wave for general viscous gas, Journal of Differential Equations, https://doi.org/10.1016/j.jde.2020.01.004. 
\bibitem{HXY2008}
\newblock F.  Huang, Z.  Xin and T. Yang,
 \emph{Contact discontinuities with general perturbation for gas
	motion},
\newblock Adv. Math., \textbf{219} (2008), 1246--1297.
\bibitem{IO1960} A. M. Ilin and O. A. Oleinik, \emph{Asymptotic behavior of solutions of the cauchy problem for some quasi-linear equations for large values of the time}, Mat. Sb. (N.S.) 51 (93) (1960), 191-216.
\bibitem{KM1985} S. Kawashima and A. Matsumura, \emph{Asymptotic stability of traveling wave solutions of systems for one-dimensional gas motion}, Comm. Math. Phys.
Vol. 101, No. 1 (1985), 97-127.
\bibitem{KMN1985} S. Kawashima, A. Matsumura and K. Nishihara, \emph{Asymptotic behavior of solutions for the equations of a viscous heat-conductive gas}, Proc. Japan Acad. Ser. A Math. Sci.
Vol. 62, No. 7 (1986), 249-252.
\bibitem{KNM2004} S. Kawashima, S. Nishibata and M. Nishikawa, \emph{$L^p$ energy method for multi- dimensional viscous conservation laws and application to the stability of planar waves}, J. Hyperbolic Differ. Equations 01 (2004), no. 03, 581-603.
\bibitem{Lax1957} P. Lax, \emph{Hyperbolic systems of conservation laws II}, Communications on Pure and Applied
Mathematics 10 (1957), no. 4, 537-566.
\bibitem{Liu1978} T.  Liu, \emph{Invariants and asymptotic behavior of solutions of a conservation law}, Proceedings
of the American Mathematical Society 71 (1978), no. 2, 227.
\bibitem{LX1988} T.  Liu and Z. Xin, \emph{Nonlinear Stability of Rarefaction Waves for Compressible
Navier Stokes Equations}, Communications in Mathematical Physics 118 (1988), 451-465.
\bibitem{LX1997} T.  Liu and Z. Xin, \emph{ Pointwise decay to contact discontinuities for systems of viscous conservation laws}, Asian J. Math. 1 (1997) 34-84.
\bibitem{LZ2009} T.  Liu and Y.  Zeng, \emph{ Time-asymptotic behavior of wave propagation around a viscous shock profile}, Comm. Math. Phys. 290 (2009), no. 1, 23-82.

\bibitem{MN1985} A. Matsumura and K. Nishihara, \emph{On the stability of traveling wave solutions of a one-dimensional model system for compressible viscous gas}, Japan J. Appl. Math., 2 (1985), 17-25.

\bibitem{MN1986} A. Matsumura and K. Nishihara, \emph{Asymptotics toward the rarefaction waves of the solutions of a one-dimensional model system for compressible viscous gas}, Japan Journal of Applied
Mathematics 3 (1986), no. 1, 1-13.

\bibitem{Nirenberg1959} L. Nirenberg, \emph{On elliptic partial differential equations}, Ann. della Sc. Norm. Super. di Pisa,
Cl. di Sci. 2 (1959), no. 13, 115-162.

\bibitem{Oleinik1957} O. A. Oleinik, \emph{Discontinuous solutions of non-linear differential equations}, Uspehi Mat. Nauk
(N.S.) 12 (1957), no. 3(75), 3-73.

\bibitem{Serre2002} D. Serre, \emph{ {$ L^1 $}-stability of nonlinear waves in scalar conservation laws}, Handbook of Differential Equations: Evolutionary Equations, (2002),473-553.

\bibitem{Smoller1994} J. Smoller, \emph{Shock Waves and Reaction Diffusion Equations}, vol. 258, Springer-Verlag, New
York-Berlin, 1994.

\bibitem{SX1993} A. Szepessy, Z. Xin, \emph{Nonlinear stability of viscous shock waves}, Arch. Ration. Mech. Anal. 122 (1993) 53-103.

\bibitem{WW2019} T. Wang and Y. Wang, \emph{Nonlinear stability of planar rarefaction wave to the three-dimensional Boltzmann equation}, Kinetic and Related Models,  
Vol. 12, No. 3, 2019, 637-679.

\bibitem{Xin1990} Z.  Xin, \emph{Asymptotic stability of planar rarefaction waves for viscous conservation laws in several
dimensions}, Trans. Amer. Math. Soc. 319 (1990), no. 2, 805-820.

\bibitem{XYY2019-1} Z.  Xin, Q. Yuan and Y. Yuan, \emph{Asymptotic stability of shock profiles and rarefaction waves under periodic perturbations for 1-D convex scalar viscous conservation laws}, arxiv:1902.09772 (2019), 1-42.

\bibitem{XYY2019-2} Z.  Xin, Q. Yuan and Y. Yuan, \emph{Asymptotic stability of shock waves and rarefaction waves under periodic perturbations for 1-d convex scalar conservation laws}, SIAM J. Math. Anal. 51 (2019), no. 4, 2971-2994.

\bibitem{YY2019} Q. Yuan and Y. Yuan, \emph{On Riemann solutions under different initial periodic perturbations
at two infinities for 1-d scalar convex conservation laws}, J. Differ. Equ. (2019), 1-15.



\end{thebibliography}
\end{document}